\documentclass{amsart}
\usepackage{amsfonts}
\usepackage{color}
\def\R{{\mathbb {R}}}
\def\N{{\mathbb {N}}}

\def\vp{\varphi}

\def\ve{\varepsilon}
\def\cd{\rightharpoonup}

\def\supp{\operatorname {\text{\rm supp}}}

\def\dist{\operatorname {\text{\rm dist}}}

\newtheorem{teo}{Theorem}[section]
\newtheorem{lema}[teo]{Lemma}
\newtheorem{prop}[teo]{Proposition}

\theoremstyle{remark}

\theoremstyle{definition}

\numberwithin{equation}{section}

\begin{document}
	
\title[Extremals in HLS inequalities for  stable processes]{Extremals in Hardy-Littlewood-Sobolev inequalities for  stable processes}
\author[A. de Pablo, F. Quir\'os and A. Ritorto]{Arturo de Pablo, Fernando Quir\'os and Antonella Ritorto}

\address[A. de Pablo]{Departamento de Matem\'{a}ticas, Universidad Carlos III de Madrid, 28911 Legan\'{e}s, Spain.}
\email{arturop@math.uc3m.es}

\address[F. Quir\'os]{Departamento de Matem\'{a}ticas, Universidad 	Aut\'{o}noma de Madrid, 28049 Madrid, Spain.}
\email{fernando.quiros@uam.es}

\address[A. Ritorto]{Mathematisch-Geographische Fakult\"at, Katholische Universit\"at Eichst\"att-Ingolstadt, 85071 Eichst\"att, Germany}
\email{antonella.ritorto@ku.de}


\begin{abstract}
We prove the existence of an extremal function in the Hardy-Littlewood-Sobolev inequality for the energy associated to an stable operator. To this aim we obtain a concentration-compactness principle for stable processes in $\R^N$.
\end{abstract}

\date{\today}

\maketitle

\section{Introduction}
After the seminal work of Brezis and Nirenberg~\cite{Brezis-Nirenberg}, the relation between the existence of positive solutions of elliptic equations with critical nonlinearities perturbed by lower order terms and the existence of an extremal function for the Sobolev type inequality associated with the elliptic operator is now well understood. The aim of this paper is to prove the existence of such an extremal function for operators generated by symmetric stable L\'evy processes, as a first step   to study the corresponding critical problems.

L\'evy processes appear in several applied fields, as for example biology, mathematical finance and physics, to take into account the possible appearance of jumps, and have been analyzed thoroughly in the last years from the point of view both of analysis and probability; see for instance the survey \cite{RosOton2016} and the references therein.  L\'evy operators have the form
\begin{equation}\label{operator-L}
\displaystyle\mathcal{L}u(x)= \textrm{P.V.}\int_{\R^{N}}(u(x)-u(x-y))\,\nu(x,dy),
\end{equation}
where P.V. stands for principal value, and the nonnegative measure $\nu$ satisfies
\begin{equation}\label{kernel-L}
\displaystyle\int_{\mathbb{R}^{N}} \min\{1,\,|y|^2\}\, \nu(x,dy)\le C<\infty \quad \text{for all } x\in\mathbb{R}^N.
\end{equation}
They play the same role for L\'evy processes as the Laplacian does in the description of Brownian motion. If $\mathcal{L}$ is generated by a process which is moreover symmetric and stable, that is, a process $X=\{X_t\}_{t\ge0}$ satisfying
$$
\gamma X_t= X_{\gamma^\alpha t},\quad \gamma>0,\; t\ge0,
$$
then
\begin{equation}
  \label{a-stable-measure}\tag{S}
  \nu(x,dy)=\nu(dy)=\dfrac{dr}{r^{1+\alpha}}\,d\mu(\theta),\qquad r=|y|,\;\theta=\frac{y}{|y|},
\end{equation}
for some measure $d\mu(\theta)$ on the sphere known as the \emph{spectral measure}. This measure  is required to be finite and to  satisfy a non-degeneracy assumption; see Section~\ref{sect:preliminaries} for the details. The singularity at the origin of the kernel $\nu$, which  defines the differential character of the operator, makes it of order $\alpha$, like $\alpha$ derivatives. Observe that condition~\eqref{operator-L} implies $0<\alpha<2$. We will restrict ourselves to the case of several spatial dimensions, $N\ge2$, since the only stable operators in one dimension are $c(-\partial^2_{xx})^{\frac\alpha2}$.

When $d\mu(\theta)=d\theta$ we recover (a multiple of) the  fractional Laplacian $(-\Delta)^{\frac\alpha2}$, the best known case of an $\alpha$-stable operator (stable operator of order $\alpha$). The general case is more involved since the operator may be:
\begin{itemize}
\item \emph{anisotropic}, if  $\mu$ is not rotationally invariant in $y$;
\item \emph{rough}, meaning that $\nu$ need not be smooth outside the origin $y=0$.
\end{itemize}
In addition, the spectral measure may vanish in some directions and be singular for others; this is the case, for instance, when $\mathcal{L}$ is the sum of fractional Laplacians of order $\alpha$ of smaller dimensions, a well-known example of an anisotropic $\alpha$-stable operator. On the positive side we have the homogeneity of operators in this class, which plays to our advantage.

Associated with the $\alpha$-stable operator $\mathcal{L}$ we have the semilinear elliptic problem
\begin{equation}\label{intro-problem}
\mathcal{L}u = f(u)\quad \text{in } \Omega,\qquad
u=0 \quad\text{in } \Omega^c=\R^N \setminus \Omega,
\end{equation}
where $\Omega\subset\mathbb{R}^N$ is a bounded set.  A case of special interest arises when the nonlinearity $f(u)$ is critical in some sense related to the operator $\mathcal{L}$; see below.

In principle, solutions to~\eqref{intro-problem} need not have the required regularity for $\mathcal{L}u$ to be well defined. In particular the pointwise expression \eqref{operator-L} makes sense only for functions that do not grow too much at infinity and are H\"older continuous $C^{\alpha+\varepsilon}$  for some $\varepsilon>0$.  We then have to work in general with weak solutions
 by using the bilinear form
\begin{equation}
  \label{bilinear}
  \mathcal{E}(u,v)=\frac12\int_{\mathbb{R}^N}\int_{\mathbb{R}^N}\left(u(x)-u(x-y)\right)
  \left(v(x)-v(x-y)\right)\nu(dy)\,dx,
\end{equation}
defined for functions in the energy space
\begin{equation}
  X(\Omega)=\left\{ u\colon \R^N \to \R \text{ measurable} \colon u\equiv0 \text{ on } \Omega^c \, \text{ and } \mathcal{E}(u,u)<\infty \right\},
\end{equation}
endowed with the norm $\| u\|_{X(\Omega)}=\mathcal{E}(u,u)^{\frac{1}{2}}$.  Under our assumptions on the spectral measure, this norm is equivalent to the standard fractional Sobolev norm in $\dot H^{\frac\alpha2}(\R^N)$; see~\eqref{equivalencia-normas}. Note that $X(\Omega)$ takes into account the (nonlocal) boundary condition in problem~\eqref{intro-problem}.

Assuming the symmetry property
\begin{equation}
\label{symmetry}
 d\mu(\theta)=d\mu(-\theta), \quad\theta\in\mathbb{S}^{N-1},
\end{equation}
then
\begin{equation}
  \label{weak}
  \mathcal{E}(u,v)=\int_{\mathbb{R}^N}\mathcal{L}u\,v \quad\text{for every }u,\,v\in C_0^\infty(\mathbb{R}^N).
\end{equation}
We then say that $u\in X(\Omega)$ is a (weak) solution to problem~\eqref{intro-problem}  if
$$
\mathcal{E}(u,\vp)=\int_\Omega f(u)\vp \quad\text{for every }\vp\in X(\Omega).
$$

Now, due to  the compact embedding $X(\Omega)\hookrightarrow L^{r}(\Omega)$ when $r<2^*_\alpha:=\frac{2N}{N-\alpha}$,  existence of a solution to problem~\eqref{intro-problem} follows by means of standard variational techniques provided $f$ is subcritical in the following sense,
\begin{equation}
  \label{f-subcrit}
  |f(u)|\le c(1+|u|)^\gamma,\qquad  \gamma<p=2^*_\alpha-1=\frac{N+\alpha}{N-\alpha}\,.
\end{equation}
In fact the existence of a solution  is equivalent to the existence of a solution for a minimization
problem, where compactness has a preeminent role.
When $N\le \alpha$ there is no critical exponent and the embedding is always compact. Note that in our context we have $N\ge2>\alpha$.

On the other hand,  it was proved in \cite{RosOton-Serra} by means of a Pohozaev identity that problem~\eqref{intro-problem} with $\mathcal{L}=(-\Delta)^{\frac\alpha2}$ has no positive solution if $f(u)=u^\gamma$, $\gamma\ge p$, and $\Omega$ is star-shaped. Later a similar result was shown to hold also in the $\alpha$-stable case; see \cite{RosOton-Serra-Valdinoci}.

A natural question is whether a perturbation of the critical reaction may yield existence of a solution, for instance for the problem
\begin{equation}\label{problem-cc0}
\mathcal{L} u = u^{p}+\lambda u^q \quad\text{in } \Omega,\qquad
u=0 \quad\text{in } \Omega^c,\qquad u>0 \quad \text{in }\Omega,
\end{equation}
where $\lambda>0$ and $q<p$. This question was addressed  for the  local problem, where $\mathcal{L}=-\Delta$, and for all $0<q<p$ in the classical papers \cite{Ambrosetti-Brezis-Cerami-94, Brezis-Nirenberg}, and extended later to the case $\mathcal{L}=(-\Delta)^{\frac\alpha2}$ in \cite{Barrios-Colorado-Servadei-Soria, FB-Saintier-Silva, Servadei-Valdinoci}. It has even been studied for the related case of the so-called \emph{spectral fractional Laplacian}, defined by means of the eigenvalues of the Laplacian in $\Omega$, in~\cite{Brandle-Colorado-dePablo-Sanchez, Tan2011}; see also \cite{Colorado-dePablo-Sanchez}.

For the critical exponent $2^*_\alpha$ we still have the inclusion $X(\Omega)\hookrightarrow L^{2^*_\alpha}(\Omega)$, which follows from the Hardy-Littlewood-Sobolev (HLS) inequality
\begin{equation}
  \label{HLS}
\|u\|_{\dot H^{\frac\alpha2}}\ge S\|u\|_{2^*_\alpha}\,,
\end{equation}
but it is not compact. As shown in~\cite{Brezis-Nirenberg}, lack of compactness can be circumvented in perturbed critical cases like \eqref{problem-cc0} if there is an extremal for the corresponding (HLS) inequality. However, except in particular situations for which they are explicit, proving the existence of an extremal already requires having some compactness, which is not immediate in the whole space where concentration effects may occur. In the outstanding series of works \cite{Lions-locally-PartI}--\cite{Lions-limit-case-PartII} P.\,L. Lions developed  a general method, based on what he named as concentration-compactness principle (CCP), to deal with such problems. His papers include the case of the fractional Laplacian as well as the $p$-fractional Laplacian.

In our case the operator is anisotropic, so extremals for the (HLS) inequality
\begin{equation}\label{ineq}
\mathcal{E}(u,u) \ge S \|u\|_{2_{\alpha}^*}^2 \quad \text{ for every } u \in \dot{H}^{\frac{\alpha}{2}}(\R^N),
\end{equation}
$S$ being the best constant (see \eqref{S}), are not expected to be radial, and there is no hope to get them explicitly. Thus, in order to prove their existence we will follow the approach of P.\,L. Lions, establishing the CCP.

The purpose of this work is to prove a CCP for $\alpha$-stable operators, Theorem~\ref{concentracion-compacidad-RN}, and then to apply it to show the existence of an  extremal for inequality~\eqref{ineq}, Theorem \ref{existencia-extremal}. Our results can be seen as a first step to study the perturbed problem~\eqref{problem-cc0}, which will be the subject of a future work.

We remark that our proof of the CCP can be easily  adapted to deal with operators having kernels which are comparable to the one of the fractional Laplacian in the following sense:
\begin{equation}
\label{eq:comparable.kernel}
  \nu(x,dy)=\frac{a(x,y)\,dy}{|y|^{N+\alpha}},\qquad  0<c_1\le a(x,y)\le c_2<\infty.
\end{equation}
However, the existence of an extremal for \eqref{ineq} in the case of stable operators uses in an essential way the homogeneity of the operator, and hence cannot be adapted to general kernels of the form~\eqref{eq:comparable.kernel}.  On  the other hand, in the frame of $\alpha$-stable processes the main issue to deal with relies on its anisotropic feature. To overcome this, we provide a geometric result that gives an accurate decay of certain integrals involving the measure $\mu$.

\noindent \textit{Notation.} Throughout the paper, the letters $c,\,C$ denote generic positive constants; $B_R(a)$ denotes the ball $\{x\in\mathbb{R}^N:|x-a|<R\}$, and $B_R=B_R(0)$; the $L^r$-norms are written as $\|\cdot\|_r$, without specifying the domain when no confusion arises.

\section{Preliminaries}
\label{sect:preliminaries}

We first show some properties of the bilinear form~\eqref{bilinear}. Using Fourier Transform we see that
$$
\mathcal{E}(u,u)=\int_{\mathbb{R}^N}m(\xi)|\hat u(\xi)|^2\,d\xi,
$$
where
\begin{equation}
\label{multiplier}
\displaystyle m(\xi)=|\xi|^\alpha g(\xi/|\xi|), \qquad g(\zeta)=c_{N,\alpha}\int_{\mathbb {S}^{N-1}}|\zeta\cdot \theta|^\alpha\,d\mu(\theta),
\end{equation}
for some constant $c_{N,\alpha}>0$.
Therefore $m(\xi)$ is homogeneous.
If $\mu$ is finite, $\mu(\mathbb{S}^{N-1})<\infty$, and satisfies
\begin{equation}\label{eq.ellipcity}
 \inf_{\zeta\in \mathbb{S}^{N-1}} \int_{\mathbb{S}^{N-1}}|\zeta\cdot\theta|^{\alpha}\,d\mu(\theta)\ge c_0>0,
\end{equation}
we obtain also $m(\xi)\sim|\xi|^\alpha$, from which we deduce the required equivalence
\begin{equation}\label{equivalencia-normas}
\mathcal{E}(u,u)\sim\int_{\mathbb{R}^N}|\xi|^\alpha|\hat u(\xi)|^2\,d\xi=\|u\|_{\dot H^{\frac\alpha2}}^2.
\end{equation}
By $f\sim g$ we understand that there exist two positive constants $c_1,c_2$ such that $c_1\le f/g\le c_2$.

When the spectral measure is uniform,  $d\mu(\theta)=\gamma\,d\theta$ with $\gamma>0$ a constant, we have $m(\xi)=c_{N,\alpha} \gamma|\xi|^{\alpha}$, and the operator reduces to a multiple of the fractional Laplacian. But the class of stable operators is much more general, and includes, for example, the relevant case  in which $\mathcal{L}$ is the sum of fractional Laplacians of smaller dimensions. A particular instance of such case is
\begin{equation*}
\label{sumalaplacianos}
\mathcal{L}=\sum_{j=1}^Nb_j(-\partial^2_{x_jx_j})^{\frac\alpha2},
\end{equation*}
for which the spectral measure  is $d\mu(\theta)=\frac12c_{1,\alpha}\sum_{j=1}^Nb_j(\delta(\theta-e_j)+\delta(\theta+e_j))$, $\{e_j\}$ being the canonical basis in $\R^{N}$. This is in some sense the extreme case of spectral measure, a sum of Dirac deltas. Here the multiplier is
$m(\xi)=\sum_{j=1}^Nb_j|\xi_j|^{\alpha}$.

We now define, for a measure $\nu$,  the $\nu$--gradient
\begin{equation}
  \label{nu-gradient}
  D^\nu u(x)=\left(\frac12\int_{\mathbb{R}^N}\left|u(x)-u(x-y)\right|^2\nu(dy)\right)^{1/2}.
\end{equation}
If the measure corresponds to an $\alpha$-stable process,  this is written in terms of the spectral measure as
\begin{align*}
D^\nu u(x)= \frac{1}{2}\int_{0}^{\infty} \int_{\mathbb{S}^{N-1}} |u(x)-u(x-r\theta)|^2 \, d\mu(\theta) \frac{dr}{r^{1+\alpha}}.
\end{align*}
Observe that, by definition,
$$
\mathcal{E}(u,u)= \|D^{\nu} u\|_{2}^2.
$$
The following scaling properties are straightforward.
\begin{lema}\label{scaling} Let $u_{\tau}(x)=u( x/\tau)$, $\tau>0$. We have
$$
D^{\nu}u_\tau(x)= \tau^{-\frac\alpha2}D^{\nu}u( x/\tau),
$$
and as a consequence
$$
\|D^{\nu}u_\tau\|_2= \tau^{\frac{N-\alpha}2}\|D^{\nu}u\|_2,\qquad
\|D^{\nu}u_\tau\|_{\frac{2N}\alpha}=\|D^{\nu}u\|_{\frac{2N}\alpha}.
$$
\end{lema}


\

Next lemma, whose proof is also standard, shows to be useful in later calculations.
\begin{lema}\label{decay} If $u\in C^{2}_c(\R^N)$ then $D^{\nu}u \in L^\infty(\mathbb{R}^N)\cap L^2(\mathbb{R}^N)$. Moreover
$$
\|D^{\nu}u\|_\infty\le C(\|u\|_{\infty}^2+\|\nabla u\|_{\infty}^2)^{1/2},\qquad\|D^{\nu}u\|_2\le C (\|u\|_{2}^2+\|\nabla u\|_{2}^2)^{1/2},
$$
where $C$ depends on the support of $u$.
\end{lema}
\begin{proof} Just decompose the integral defining $D^\nu u$ into two parts, for small $|y|$ and large $|y|$.
\end{proof}

In particular this implies  $C^{2}_c(\Omega)\subset X(\Omega)$.

As to the action of the bilinear form on a product we prove the following result.
\begin{lema}\label{lem.product}
Let $u,\,v\in \dot H^{\frac\alpha2}(\mathbb{R}^N)$. Then,
\begin{equation}
  \label{eq.product}
  \mathcal{E}(uv,uv)^{\frac{1}{2}}\le \left(\int_{\mathbb{R}^N}|u|^2|D^\nu v|^2\,dx\right)^{\frac{1}{2}} + \left(\int_{\mathbb{R}^N}|D^\nu u|^2|v|^2\,dx \right)^{\frac{1}{2}}.
\end{equation}
\end{lema}
\begin{proof} The proof uses Minkowski  inequality and the symmetry condition \eqref{symmetry},
\begin{align*}
\Big(&\int_{\R^N}|D^\nu (uv)(x)|^2\, dx\Big)^{\frac{1}{2}} \\ =&\left(\int_{\R^N}\frac12\int_{\mathbb{R}^N}\left|u(x)v(x)-u(x-y)v(x-y)\right|^2\nu(dy)\, dx\right)^{\frac{1}{2}} \\
\le& \left(\frac12\int_{\R^N}\int_{\mathbb{R}^N}\left|u(x)(v(x)-v(x-y))+v(x-y)(u(x)-u(x-y))\right|^2\nu(dy)\, dx\right)^{\frac{1}{2}} \\
\le& \left(\frac12\int_{\R^N}\int_{\R^N}|u(x)|^2|v(x)-v(x-y)|^2 \, \nu(dy)\, dx\right)^{\frac{1}{2}}\\ &+\left(\frac12\int_{\R^N}\int_{\mathbb{R}^N}|v(x-y)|^2|u(x)-u(x-y)|^2\nu(dy)\, dx\right)^{\frac{1}{2}}\\
=& \left(\int_{\mathbb{R}^N}|u(x)|^2|D^\nu v(x)|^2\,dx\right)^{\frac{1}{2}} + \left(\int_{\mathbb{R}^N}|v(w)|^2|D^\nu u(w)|^2\,dw \right)^{\frac{1}{2}}.
\end{align*}
\end{proof}

In the last result in this preliminary section we control the $L^2$-norm  of functions in $\dot{H}^{\frac\alpha2}(\R^N)$ in a very precise way by interpolation, following ideas of \cite[Lemma~3.2]{Marano-Mosconi}.

\begin{lema}\label{lemaMM-adaptado1}
Let $u\in \dot{H}^{\frac\alpha2}(\R^N)$ and $R>0$. There exists a constant $C>0$ such that
$$
\int_{B_R} |u|^2 \, dx \le C  R^{\alpha}\left[ \int_{B_{4R}-B_{2R}}|D^{\nu}u|^2 \, dx+\left(\int_{B_{4R}-B_{2R}} |u|^{\frac N{N-\alpha}}\,dx\right)^{\frac{N-\alpha}N}\right].
$$
\end{lema}
\begin{proof}
Exactly as in that paper we can obtain the estimate
$$
\int_{B_R} |u|^2 \, dx \le C  \int_{B_{4R}-B_{2R}}\left(R^{\alpha}|D^{\nu}u|^2+ |u|^2 \right) \, dx.
$$
Now apply H\"older  inequality to the integral of the second term.
\end{proof}

\section{Concentration-compactness principle}\label{sect:CCP}

In this section we prove a version of the CCP of P.L. Lions. We follow \cite{Lions-limit-case-PartI}, but we also use some ideas of the proof performed in \cite{FB-Saintier-Silva} for the case of the fractional $p$-Laplacian.

Consider  the best constant of the (HLS) inequality,
\begin{equation}\label{S}
\begin{array}{rl}S&\displaystyle= \inf_{\stackrel {u\in \dot{H}^{\frac\alpha2}(\R^N)}{u\neq 0}} \frac{\mathcal{E}(u,u)}{\|u\|_{2_{\alpha}^*}^2}
\\ [3mm]&\displaystyle= \inf\left\{ \mathcal{E}(u,u) \colon u \in \dot{H}^{\frac\alpha2}(\R^N), \, \int_{\R^N}|u|^{2_{\alpha}^*} \, dx = 1 \right\}.
\end{array}\end{equation}
Denote also $\mathcal{M}(\R^N)$ the set of nonnegative finite Radon measures on $\R^N$.

\begin{teo}\label{concentracion-compacidad-RN}
Let $\{ u_k\}_{k\in \N}$ be a sequence in $\dot{H}^{\frac\alpha2}(\R^N)$ such that $u_k\cd u$ in $\dot{H}^{\frac\alpha2}(\R^N)$ as $k\to \infty$.	
Then there exist two measures $\varrho$ and $\eta$, an at most countable set $I$ and  positive numbers $\{\varrho_i\}_{i\in I}, \{\eta_i\}_{i\in I}$ such that
\begin{eqnarray}
\label{medida-gradiente}
|D^{\nu}u_k|^2dx \stackrel{*}{\cd} \eta \ge |D^{\nu}u|^2dx + \sum_{i\in I} \eta_i \delta_{x_i} \text{ in } \mathcal{M}(\R^N),\\
\label{medida-critico}
|u_k|^{2_{\alpha}^*}dx \stackrel{*}{\cd} \varrho=|u|^{2_{\alpha}^*}dx +\sum_{i\in I} \varrho_i \delta_{x_i} \text{ in } \mathcal{M}(\R^N),
\end{eqnarray}
as $k\to \infty$, where for every $i\in I$ it holds
\begin{equation}
\label{eq:relation.coeficients}
S^{\frac{1}{2}}\varrho_i^{\frac{1}{2_{\alpha}^*}} \le \eta_i^{\frac{1}{2}}.
\end{equation}
Moreover,
\begin{equation}\label{masa-infinito}
\lim_{k\to\infty}\int_{\R^N}|D^{\nu}u_k|^2dx = \eta(\R^N) + \eta_{\infty}, \qquad \lim_{k\to\infty} \int_{\R^N} |u_k|^{2_{\alpha}^*} dx = \varrho(\R^N) + \varrho_{\infty},
\end{equation}
with
$$
\eta_{\infty} = \lim_{R\to \infty} \limsup_{k\to \infty} \int_{|x|>R}|D^{\nu}u_k|^2dx, \qquad
\varrho_{\infty} =  \lim_{R\to \infty} \limsup_{k\to \infty} \int_{|x|>R} |u_k|^{2_{\alpha}^*} dx.
$$
\end{teo}
\begin{proof}
\emph{Step 1}. We may assume $u=0$; the general case follows, as in \cite{Lions-limit-case-PartI}, from Brezis-Lieb Lemma \cite[Theorem 1]{Brezis-Lieb}.
	
\noindent\emph{Step 2}. We show that the measures $\varrho$ and $\eta$ verify the reverse H\"older inequality
\begin{equation}\label{reverse-Holder}
S^{\frac{1}{2}} \left(\int_{\R^N} |\vp|^{2_{\alpha}^*} \, d\varrho \right)^{\frac{1}{2_{\alpha}^*}} \le \left( \int_{\R^N} |\vp|^2 \, d\eta \right)^{\frac{1}{2}}\quad\text{for every }\vp \in C_c^{\infty}(\R^N).
\end{equation}	
Let $\vp \in C_c^{\infty}(\R^N)$. Then by the (HLS) inequality we get
$$
S^{\frac{1}{2}} \|\vp u_k\|_{2_{\alpha}^*}\le \mathcal{E}(\vp u_k, \vp u_k)^{\frac{1}{2}} \qquad \text{ for every } k\in\N.
$$
Since $|u_k|^{2_{\alpha}^*}dx \stackrel{*}{\cd} \varrho$  as $k\to\infty$, it follows that
\begin{equation}
\label{la-de-nu}
\limsup_{k \to \infty}\|\vp u_k\|_{2_{\alpha}^*}= \left(\int_{\R^N} |\vp|^{2_{\alpha}^*} \, d\varrho \right)^{\frac{1}{2_{\alpha}^*}}.
\end{equation}
By Lemma~\ref{lem.product}, for every $k\in\N$ it holds
$$
\mathcal{E}(\vp u_k, \vp u_k)^{\frac{1}{2}}\le \left(\int_{\R^{N}} |u_k|^2 |D^{\nu} \vp|^2\, dx\right)^{\frac{1}{2}} +   \left(\int_{\R^{N}} |\vp|^2 |D^{\nu} u_k|^2  \, dx \right)^{\frac{1}{2}}.	
$$	
Now, since $|D^{\nu}u_k|^2dx \stackrel{*}{\cd} \eta$ as $k\to\infty$, we get
\begin{equation}
\label{la-de-mu}
\limsup_{k\to \infty} \int_{\R^{N}} |\vp|^2 |D^{\nu} u_k|^2  \, dx = \int_{\R^{N}} |\vp|^2 \, d\eta.
\end{equation}
Therefore, to finish the proof of the reverse H\"older inequality, it remains to show that
\begin{equation}\label{termino-malo}
\limsup_{k\to \infty}\int_{\R^{N}} |u_k|^2|D^{\nu} \vp|^2\, dx=0.
\end{equation}
Let $R>0$ to be determined. Then,
$$
\int_{\R^{N}} u_k^2 |D^{\nu} \vp|^2\, dx=\underbrace{\int_{B_R}|u_k|^{2} |D^{\nu}\vp|^2\, dx}_{I}+\underbrace{\int_{B^c_R}|u_k|^{2} |D^{\nu}\vp|^2\, dx}_{II}.
$$
By using H\"older  inequality with $q=\frac{2_{\alpha}^*}{2}=\frac N{N-\alpha}$, $q'= \frac{N}{\alpha}$, (HLS) inequality and the boundedness of $\{u_k\}_{k\in \N}$ in $\dot{H}^{\frac\alpha2}(\R^N)$, we get
$$
II\le \|u_k\|_{2_{\alpha}^*}^2 \left( \int_{B^c_R} |D^{\nu} \vp|^{\frac{2N}{\alpha}}\, dx\right)^{\frac{\alpha}{N}}
\le C\left( \int_{B^c_R} |D^{\nu} \vp|^{\frac{2N}{\alpha}}\, dx\right)^{\frac{\alpha}{N}}\to 0,
$$
as $R\to\infty$ since $D^{\nu} \vp\in L^{\frac{2N}{\alpha}}(\R^N)$.
Choose then  $R>2$ large such that $II<\ve$ for every $k\in \N$.
	
Now, since  $H^{\frac\alpha2}(B_R)\hookrightarrow L^2(B_R)$ compactly we have
$$
I\le C\int_{B_R}|u_k|^2\, dx\to 0,
$$
as $k\to \infty$. That finishes the proof of \eqref{termino-malo}.
The reverse H\"older  inequality \eqref{reverse-Holder} follows from \eqref{la-de-nu}, \eqref{la-de-mu} and \eqref{termino-malo}.
	
\noindent\emph{Step 3}. From \eqref{reverse-Holder} it follows exactly as in \cite{Lions-limit-case-PartI}, that there exist a countable set $I$, points $\{ x_i\}_{i\in I} \subset \R^N$ and positive numbers $\{\eta_i\}_{i\in I},\{\mu_i\}_{i\in I}$ such that
\begin{equation}\label{mu-nu-deltas}
\varrho=\sum_{i\in I}\varrho_i \delta_{x_i}, \qquad \eta\ge \sum_{i\in I}\eta_i \delta_{x_i}.
\end{equation}
	
\noindent\emph{Step 4}. Fix $i\in I$, and without loss of generality assume $x_i=0$. Let $\vp \in C_c^{\infty}(\R^N)$ be such that $0\le \vp\le 1, \vp(0)=1$ and $\supp \vp = B_1$, and consider the function $\vp_{\ve}(x)=\vp(x/\ve)$ for $0<\ve<1$.
	
Proceeding in the same way as we did when proving the reverse H\"older inequality, for every $k\in\N$ and $0<\varepsilon<1$, we obtain for the product $\vp_{\ve}u_k$
$$
S^{\frac{1}{2}} \left(\int_{\R^N} |\vp_{\ve}|^{2_{\alpha}^*} \, d\varrho \right)^{\frac{2}{2_{\alpha}^*}} \le  \left(\int_{\R^N}|\vp_{\ve}|^2 \, d \eta \right)^{\frac{1}{2}}+\left(\int_{\R^N}|u_k|^2|D^{\nu} \vp_{\ve}|^2\, dx \right)^{\frac{1}{2}},
$$
which gives
\begin{equation}\label{coeficientes-epsilon}
S^{\frac{1}{2}}\varrho_i^{\frac{1}{2_{\alpha}^*}} \le (\eta\left(B_{\ve}\right))^{\frac{1}{2}}+ \left(\int_{\R^N}|u_k|^2|D^{\nu} \vp_{\ve}|^2\, dx\right)^{\frac{1}{2}}.
\end{equation}
From \eqref{mu-nu-deltas}, it is clear that $\lim\limits_{\ve\to 0}\eta\left( B_{\ve}\right)=\eta_i$.
Let us continue by proving that
\begin{equation}\label{termino-malo-epsilon}
\lim_{\ve \to 0}  \int_{\R^N}|u_k|^2|D^{\nu} \vp_{\ve}|^2 \, dx=0.
\end{equation}
For $k>0$ large to be chosen we split the integral
$$
\int_{\R^N}|u_k|^2|D^{\nu}\vp_{\ve}|^2\, dx =\underbrace{\int_{B_{k\ve}}|u_k|^2|D^{\nu}\vp_{\ve}|^2\, dx}_{A}+\underbrace{\int_{B^c_{k\ve}}|u_k|^2|D^{\nu}\vp_{\ve}|^2\, dx}_{B},
$$
and apply H\"older  inequality to both integrals with the same exponents as before, $q=\frac{N}{N-\alpha}, q'=\frac{N}{\alpha}$ . Using Lemma~\ref{scaling} we get
$$
A
\le c\left(\int_{B_{k\ve}} |u_k|^{\frac{2N}{N-\alpha}}\, dx \right)^{\frac{N-\alpha}{N}}\left(\int_{\R^N}|D^{\nu}\vp_{\ve}|^{\frac{2N}\alpha}\, dx\right)^{\frac\alpha N} \le c\left(\int_{B_{k\ve}} |u_k|^{\frac{2N}{N-\alpha}}\, dx \right)^{\frac{N-\alpha}{N}}.
$$
It follows that $\lim\limits_{\ve\to  0}A=0$ due to the fact that $u_k$ belongs to $L^{\frac{2N}{N-\alpha}}(\R^N)$.

On the other hand,
$$	
B\le c\left(\int_{\R^N} |u_k|^{\frac{2N}{N-\alpha}}\, dx \right)^{\frac{N-\alpha}{N}}\left(\int_{B^c_{k\ve}}|D^{\nu}\vp_{\ve}|^{\frac{2N}\alpha}\, dx\right)^{\frac\alpha N} = c\left(\int_{B^c_k} |D^{\nu}\vp|^{\frac{2N}\alpha}\, dx\right)^{\frac\alpha N}.
$$
Therefore, taking $k$ large we get $B$ small, and then, for $k$ fixed taking $\ve$ small we obtain $A+B$ small, so we conclude~\eqref{termino-malo-epsilon}. This implies the desired relation~\eqref{eq:relation.coeficients}.

\noindent\emph{Step 5}. To see \eqref{masa-infinito} we follow ideas from \cite{FB-Saintier-Silva}. Let us consider a smooth function $\Phi \colon [0, \infty) \to [0,1]$ such that $\Phi=0$ in $[0,1]$ and $\Phi=1$ in $[2,\infty)$. Given $R>0$, define $\Phi_R(x)= \Phi(|x|/R)$.

We can rewrite, for $k\in\N$ and $R>0$,
\begin{equation}\label{bla}
\int_{\R^N} |D^{\nu} u_k|^2\, dx = \int_{\R^N} |D^{\nu} u_k|^2 \Phi_R^2\, dx + \int_{\R^N} |D^{\nu} u_k|^2 (1-\Phi_R^2)\,dx.
\end{equation}
Notice that
$$
\int_{|x|>2R} |D^{\nu} u_k|^2\, dx \le \int_{\R^N} |D^{\nu} u_k|^2 \Phi_R^2\, dx \le \int_{|x|>R} |D^{\nu} u_k|^2\, dx,
$$
from where we deduce
\begin{equation}\label{eta-infty}
\eta_{\infty} = \lim_{R\to \infty} \limsup_{k\to \infty} \int_{\R^N} |D^{\nu} u_k|^2 \Phi_R^2\, dx.
\end{equation}
Analogously, we obtain
\begin{equation*}
\varrho_{\infty} = \lim_{R\to \infty} \limsup_{k\to \infty} \int_{\R^N} | u_k|^{2^*_{\alpha}} \Phi_R^{2^*_{\alpha}}\, dx.
\end{equation*}

Since $1-\Phi_R$ is a smooth function with compact support, by Dominated Convergence Theorem
$$
\lim_{R\to\infty}\lim_{k\to \infty} \int_{\R^N} |D^{\nu} u_k|^2(1- \Phi_R^2)\, dx = \lim_{R\to\infty}\int_{\R^N} (1-\Phi_R^2)\, d\eta=\eta(\R^N).
$$
This together with \eqref{bla} and \eqref{eta-infty} yields the first part of \eqref{masa-infinito}. The second part is completely analogous.
\end{proof}

\section{Existence of an extremal for $S$}\label{sect:extremal}

We recall the Dirichlet form $\mathcal{E}(u,u)$ for $\alpha$-stable operators
\begin{align*}
\mathcal{E}(u,u)= \frac{1}{2}\int_{\R^N}\int_{0}^{\infty} \int_{\mathbb{S}^{N-1}} |u(x)-u(x-r\theta)|^2 \, d\mu(\theta) \frac{dr}{r^{1+\alpha}} \, dx.
\end{align*}

Let $t>0$ and $0\le v\in L^{1}(\R^N)$ be such that $\|v\|_{1}=1$. We define the concentration function of $v$ as
\begin{equation}\label{concentration-function}
Q_v(t)= \sup_{y\in \R^N} \int_{|x-y|<t} v(x)\, dx.
\end{equation}
Notice that $0\le Q_v(t)\le 1$ for every $t>0$.

It will be most useful to recall the following key lemma \cite[Lemma I.1]{Lions-locally-PartI}, which classifies all the possible reasons of lack of compactness when we work with a sequence with fixed norm in $L^1(\R^N)$.

\begin{lema}[\cite{Lions-locally-PartI}] \label{key-lemma} Let $\{ v_k\}_{k\in \N} \subset L^1(\R^N)$ be such that $ v_k\ge0$, $\|v_k\|_{1}= 1$. Then there exists a subsequence $\{ v_{k_j}\}_{j\in \N}$ satisfying one of the three following possibilities:
\begin{itemize}
\item[(C)] {\em Compactness:} there exists a sequence of points $\{ y_j\}_{j\in \N} \subset \R^N$ such that the family $ \{v_{k_j}(\cdot +y_j)\}_{j\in \N}$ is tight, that is,
\begin{equation*}\label{tight}
\forall \, \ve>0, \, \exists \, R>0 \text{ such that } \int_{|x-y_j|<R} v_{k_j}\, dx\ge 1 -\ve.
\end{equation*}
\item[(V)] {\em Vanishing:} $\lim\limits_{j\to \infty}Q_{v_{k_j}}(t)=0$, for every $t>0$.
\item[(D)] {\em Dichotomy:} there exists a constant $\varrho \in (0,1)$ such that, for every $\ve>0$ there exist $j_0\in \N$ and non-negative functions $v_j^1,v_j^2 \in L^1(\R^N)$ such that for every $j\ge j_0$,
\begin{equation*}
\label{dichotomy-1}
\| v_{k_j}-(v_j^1+v_j^2)\|_{1}\le\ve,
\end{equation*}	
\begin{equation*}
\label{dichotomy-2}
\left| \int_{\R^N}v_j^1\, dx-\varrho\right|\le\ve, \quad \left| \int_{\R^N}v_j^2\, dx-(1-\varrho)\right|\le\ve,
\end{equation*}
\begin{equation*}
\label{dichotomy-3}
\lim_{j\to \infty}\dist\left( \supp v_j^1, \supp v_j^2\right)=\infty.
\end{equation*}
\end{itemize}	
\end{lema}

We introduce the following notation, for $\gamma>0$,
\begin{equation}\label{S-mu}
S_{\gamma}=\inf\left\{ \mathcal{E}(u,u) \colon  u\in \dot{H}^{\frac\alpha2}(\R^N), \, \int_{\R^N} |u|^{2_{\alpha}^*}\, dx=\gamma   \right\},
\end{equation}
with $S_{1}=S$ defined in \eqref{S}.

\begin{lema}\label{vale-siempre} With the previous notations it is $S_{\gamma}=\gamma^{\frac{N-\alpha}{N}}S$ for every $\gamma \in (0,1)$.
Moreover, $S< S_{\gamma}+S_{1-\gamma}$  for every $\gamma \in (0,1)$.
\end{lema}
\begin{proof}
Let $u\in C_c^{\infty}(\R^N)$ be such that $\int_{\R^N} |u|^{2_{\alpha}^*}\, dx=1$. Then the function $v_{\gamma}(x)= \gamma^{\frac{1}{2_{\alpha}^*}} u(x)$ satisfies $\int_{\R^N} |v_{\gamma}|^{2_{\alpha}^*}\, dx=\gamma$. Thus
$$
S_{\gamma} \le \mathcal{E}(v_{\gamma}, v_{\gamma})=\gamma^{\frac{N-\alpha}{N}}\mathcal{E}(u, u).
$$
By taking the infimum we get $S_{\gamma} \le \gamma^{\frac{N-\alpha}{N}} S$. The reverse inequality is completely analogous. Therefore $S_{\gamma} = \gamma^{\frac{N-\alpha}{N}} S$. This gives
$$
S_{\gamma}+S_{1-\gamma}= \left( \gamma^{\frac{N-\alpha}{N}}+(1-\gamma)^{\frac{N-\alpha}{N}}\right) S>S,
$$
since $\phi(t)=t^{\frac{N-\alpha}{N}}$ is strictly subadditive, that is, $\phi(a+b)<\phi(a)+\phi(b)$, for every $a,b>0$.
\end{proof}

To get rid of the dichotomy case for a minimizing sequence for the constant $S$ we introduce a geometric result. It is important due to the anisotropic character of our $\nu$-gradient.

\begin{lema}\label{lema-geometrico}
	Let $\theta \in \mathbb{S}^{N-1}$. There exists  a constant $C>0$ such that for every $s>2$ it holds
	\begin{equation*}\label{desigualdad-geometrica}
		\left|\{ \xi \in \mathbb{S}^{N-1} \colon |s\xi+r\theta|<1 \}\right| \le \frac{C}{s^{N-1}}.
	\end{equation*}
\end{lema}
\begin{proof} By means of a rotation we may assume $\theta= e_1= (1, 0, \ldots, 0)$. Consider also the spherical coordinates
$$
	\begin{aligned}
	\xi_1&= \cos\vp_1,\\
	\xi_2&= \sin\vp_1\cos\vp_2,\\
&\ \vdots	\\
	\xi_{N-1}&= \sin\vp_1\cdots \sin\vp_{N-2}\cos\vp_{N-1},\\
	\xi_N&=\sin\vp_1 \cdots \sin\vp_{N-2}\sin\vp_{N-1},
\end{aligned}
	$$
where $0\le \varphi_j\le\pi$ for $1\le j\le N-2$, and $0\le \varphi_{N-1}\le2\pi$. The corresponding Jacobian is
$$
J=(\sin\vp_1)^{N-2}\ldots (\sin\vp_{N-3})^2\sin\vp_{N-2}\le(\sin\vp_1)^{N-2}.
$$ 	
Rewrite now the condition $1>|s\xi +r \theta|^2= s^2+r^2+2sr\langle \xi, \theta \rangle$ as
	$$
	\cos\varphi_1= \langle \xi, \theta \rangle < \frac{1-(s^2+r^2)}{2sr}\le-\sqrt{1-\frac{1}{s^2}}.
	$$
This implies
$$
\pi-\varphi_1< \sigma_0=\arcsin(1/s)\le C/s.
$$
	Therefore
	$$
	\left\{ \xi \in \mathbb{S}^{N-1} \colon |s\xi+r\theta| <1 \right\} \subset
	\left\{ \xi \in \mathbb{S}^{N-1} \colon \pi-\sigma_0<\varphi_1<\pi \right\}.
	$$
Thus,
$$
		\int_{\mathbb{S}^{N-1}}\chi_{\{\xi \in \mathbb{S}^{N-1} \colon |s\xi +re_1|<1\}}(\xi)\, d\xi\le 2\pi^{N-2}\int_{\pi - \sigma_0}^{\pi} (\sin\vp_1)^{N-2}d\vp_{1}\le  \frac{C}{s^{N-1}}.
$$
\end{proof}

\begin{prop}\label{FELICIDAD}
There exists a constant $C>0$ such that for every $s>2$ it holds
	\begin{equation}\label{hipotesis-mu}
	\int_{\mathbb{S}^{N-1}}\left(\int_{\frac{s}{2}}^{2s}\int_{\mathbb{S}^{N-1}}\chi_{\{ |s\xi+r\theta|<1\}}(\theta)d\mu(\theta)\frac{dr}{r^{1+\alpha}}\right)^{\frac{N}{\alpha}} \, d\xi \le \frac{C}{s^{2N-1}}.
	\end{equation}
\end{prop}
\begin{proof}
	By using Jensen  inequality and Lemma \ref{lema-geometrico}, we calculate
	\begin{align*}
		\int_{\mathbb{S}^{N-1}}&\left(\int_{\frac{s}{2}}^{2s} \int_{\mathbb{S}^{N-1}} \chi_{ \{|s\xi+r\theta|<1 \}}(\theta) \, d\mu(\theta)\frac{dr}{r^{1+\alpha}}\right)^{\frac{N}{\alpha}} \, d\xi\\
		&=\frac{C}{s^{N}}\int_{\mathbb{S}^{N-1}}\left(\int_{\frac{s}{2}}^{2s} \int_{\mathbb{S}^{N-1}} \chi_{|s\xi+r\theta|<1 \}}(\theta) \, d\mu(\theta)\frac{s^{\alpha}dr}{r^{1+\alpha}}\right)^{\frac{N}{\alpha}} \, d\xi\\
		&\le \frac{C}{s^{N}}\int_{\mathbb{S}^{N-1}}\int_{\frac{s}{2}}^{2s} \int_{\mathbb{S}^{N-1}} \left[\chi_{|s\xi+r\theta|<1 \}}(\theta)\right]^{\frac{N}{\alpha}} \, d\mu(\theta)\frac{s^{\alpha}dr}{r^{1+\alpha}} \, d\xi\\
		&= \frac{C}{s^{N-\alpha}}\int_{\mathbb{S}^{N-1}}\int_{\frac{s}{2}}^{2s} \int_{\mathbb{S}^{N-1}} \chi_{|s\xi+r\theta|<1 \}}(\xi) \, d\xi \frac{dr}{r^{1+\alpha}} d\mu(\theta) \\
		&\le \frac{C}{s^{N-\alpha}}\int_{\mathbb{S}^{N-1}}\int_{\frac{s}{2}}^{2s} \frac{1}{s^{N-1}} \frac{dr}{r^{1+\alpha}} d\mu(\theta)=\frac{C}{s^{2N-1}}.
	\end{align*}
\end{proof}

We now prove the main result of this section, the existence of an extremal function for the constant $S$ in the (HLS) inequality, see~\eqref{S}. We follow closely the technique developed in \cite[Theorem I.1]{Lions-limit-case-PartI}.

\begin{teo}\label{existencia-extremal}
The constant $S$ in \eqref{S} is attained, that is, there exists $u_* \in \dot{H}^{\frac\alpha2}(\R^N)$ such that $\mathcal{E}(u_*,u_*)=S\|u_*\|_{2_{\alpha}^*}$.
\end{teo}
\begin{proof}
	
We prove that every minimizing sequence for $S$ in $ \dot{H}^{\frac\alpha2}(\R^N)$ is relatively compact up to a translation, which implies the existence of the extremal function $u_*$. Let then $\{ u_k\}_{k\in \N}\subset\dot{H}^{\frac\alpha2}(\R^N)$ be a  sequence satisfying $\|u_k\|_{2_{\alpha}^*}=1$, $\|D^{\nu} u_k\|_2^2=\mathcal{E}(u_k,u_k)\to S$ as $k\to\infty$. 	
The key part of the proof is to apply Lemma \ref{key-lemma} to the related sequence
$$
\rho_k= |u_k|^{2_{\alpha}^*}+|D^{\nu} u_k|^2, \qquad \tilde{\rho}_k=\frac{\rho_k}{\|\rho_k\|_{1}},
$$
demonstrating that vanishing (V) and dichotomy (D) cannot occur, see \cite[Theorem~I.2]{Lions-locally-PartII}. Observe that $\|\rho_k\|_{1}\to 1+S$ as $k\to\infty$.

\noindent\emph{Vanishing}. We can rescale $\{ u_k\}_{k\in \N}$, as follows,  $u_k^R(x)= R^{\frac{\alpha-N}{2}}u_k( x/R)$ with $R>0$. The rescaled sequence is still a minimizing sequence for $S$ since
$$
\mathcal{E}(u_k^R,u_k^R)=\mathcal{E}(u_k, u_k), \quad \|u_k^R\|_{2_{\alpha}^*}=\|u_k\|_{2_{\alpha}^*}=1.
$$
Notice that the concentration function associated with each $|u_k^R|^{2_{\alpha}^*}+|D^{\nu}u_k^R|^{2}$ is $$
Q^R_k(t)= Q_k\left(\frac{t}{R} \right)=\sup_{y\in \R^N} \frac{1}{\|\rho_k\|_{1}} \int_{B_{t/R}(y)}\left[ |u_k|^{2_{\alpha}^*} + |D^{\nu}u_k|^2\right] \, dx.
$$
	
Now we avoid vanishing by choosing a sequence $\{ R_k\}_{k\in \N}$ such that
\begin{equation}\label{anti-vanishing}
Q_k^{R_k}(1)= \frac{1}{2}.
\end{equation}
Indeed, $Q_k$ is a non-decreasing continuous function with
$$
Q_k(0)=0 \quad \text{and} \quad \lim_{t\to \infty} Q_k(t)=1,
$$
so vanishing cannot occur.
	
Rename the sequence $\{u_k^{R_k}\}_{k\in \N}$ verifying \eqref{anti-vanishing} as the new $\{u_k\}_{k\in \N}$.
	
\noindent\emph{Dichotomy}. In this case we can write  for every $k\in\N$, looking at the proof of \cite[Lemma I.1]{Lions-locally-PartI}, $u_k= u_k^1+u_k^2+v_k$, with $u_k^1= u_k \xi_k, u_k^2=u_k \eta_k$, where $\xi, \eta \in C^{\infty}(\R^N)$ are two cut-off functions such that
$$
0\le \xi, \eta \le 1, \xi=1 \text{ if } |x|\le1, \xi=0 \text{ if } |x|\ge 2, \quad \eta=1 \text{ if } |x|\ge 1 , \eta=0 \text{ if } |x|\le \frac{1}{2}.
$$
Also, for a sequence $\{z_k\}_{k\in\N}\subset\mathbb{R}^N$ we consider
$$
\xi_k(x)= \xi\left(  \frac{x-z_k}{R_1}\right), \quad \eta_k(x)= \eta\left(\frac{x-z_k}{R_k}\right),
$$
with $R_1\ge R_0$ fixed and $R_k>4R_1$ large. We  have
\begin{eqnarray}
\label{ccp1}
\left|  \bar{\varrho}- \int_{B_{R_0}(z_k)} \left(|u_k|^{2_{\alpha}^*}+ |D^{\nu} u_k|^2 \, \right) dx \right|\le \ve,
\\
\label{ccp2}
\left|  \int_{B_{R_k}(z_k)-B_{R_0}(z_k)} \left(|u_k|^{2_{\alpha}^*}+ |D^{\nu} u_k|^2 \, \right) dx  \right| \le \ve,
\end{eqnarray}
where $\bar{\varrho} \in (0,1+S)$, and the sequence of points $\{z_k\}_{k\in\N}$ is chosen so that
\begin{equation}\label{ccp3}
\varrho -\ve <\int_{B_{R_1}(z_k)} |u_k^1|^{2_{\alpha}^*} \, dx <\varrho +\ve,\qquad  Q_{{k}}(R_k)\le \bar{\varrho}+\ve,
\end{equation}
for any $\rho\in(0,1)$ given.

We recall that $\xi_k=\eta_k=0$ in $B_{R_k/2}(z_k)-B_{2R_1}(z_k)$, so that there we have $v_k= u_k$.
Now we want to estimate
$$
\mathcal{E}(u_k, u_k)-\mathcal{E}(u_k^1, u_k^1)-\mathcal{E}(u_k^2, u_k^2)=A_1 -A_2-A_3-A_4,
$$
where
$$
\begin{array}{ll}
\displaystyle A_1= \int_{\R^{N}}(1-\xi_k^2-\eta_k^2)|D^{\nu} u_k|^2\, dx,\quad & \displaystyle
A_2=\int_{\R^N}|u_k|^2 \left(|D^{\nu} \xi_k|^2 + |D^{\nu} \eta_k|^2\right)\, dx,\\[8pt]
\displaystyle
A_3=\int_{\R^N}T(\xi_k, u_k)(x)\,dx,\quad&
\displaystyle A_4=\int_{\R^N}T(\eta_k, u_k)(x)\,dx,
\end{array}
$$
with
$$
T(w, z)(x)=\int_{\mathbb{R}^{N}}w(x+y)(z(x+y)-z(x))\, z(x)(w(x+y)-w(x)) \, \nu(x,dy).
$$	

We first observe that $A_1\ge0$. On the other hand,  by H\"older  inequality,
\begin{align*}	
	|A_3| &\le  \left(\int_{B_{R_1}}|u_k|^2 |D^{\nu} \xi_k|^2 \, dx\right)^{\frac{1}{2}} \left( \int_{\R^N} |D^{\nu} u_k|^2 \, dx \right)^{\frac{1}{2}}\\
	&\le C\left(\int_{B_{R_1}}|u_k|^2 |D^{\nu} \xi_k|^2 \, dx\right)^{\frac{1}{2}},
\end{align*}
and analogously
$$
	|A_4|\le   C\left(\int_{\R^N}|u_k|^2 |D^{\nu} \eta_k|^2 \, dx\right)^{\frac{1}{2}},
$$
so that $|A_3|+|A_4| \le CA_2^{\frac{1}{2}}$. Hence, it only remains to estimate $A_2$. We have
\begin{align*}
A_2&= \left( \int_{B_{R_0}(z_k)} +\int_{B_{R_k}(z_k)-B_{R_0}(z_k)} +\int_{B_{R_k}^c(z_k)} \right) |u_k|^2 \left(|D^{\nu} \xi_k|^2 + |D^{\nu} \eta_k|^2 \right) \, dx \\
&= A_{2,1}+A_{2,2}+A_{2,3}.
\end{align*}

For $A_{2,1}$, notice that using the properties of Lemmas~\ref{scaling} and~\ref{decay}, we have
$$|D^{\nu} \xi_k|^{2} + |D^{\nu} \eta_k|^2 \le \frac1{R_1^\alpha}|D^{\nu} \xi|^{2} + \frac1{R_k^\alpha}|D^{\nu} \eta|^2 \le C.
$$
Then applying Lemma \ref{lemaMM-adaptado1} we obtain the estimate
\begin{align*}
A_{2,1} \le  &C R_0^{\alpha}\Big[ \int_{B_{4R_0}(z_k)-B_{2R_0}(z_k)}|D^{\nu}u_k|^2 \, dx
\\ &+ \Big(\int_{B_{4R_0}(z_k)-B_{2R_0}(z_k)} |u|^{\frac{2N}{N-\alpha}} \, dx\Big)^{\frac{N-\alpha}N}  \Big]\le C( \ve + \ve^{\frac{N-\alpha}N}  ),
\end{align*}
since $B_{4R_0}(z_k)-B_{2R_0}(z_k)\subset B_{R_k}(z_k)-B_{R_0}(z_k)$.

As to $A_{2,2}$,    using again the cited Lemmas we have
$
\|D^{\nu} \xi_k\|_{\frac{2N}\alpha}+\|D^{\nu} \eta_k\|_{\frac{2N}\alpha} \le C$. Thus, applying H\"older  inequality and \eqref{ccp2} we obtain
$$
\begin{aligned}
A_{2,2} &\le  C \left(\int_{B_{R_k}(z_k)-B_{R_0}(z_k)}|u_k|^{\frac{2N}{N-\alpha}}\,dx\right)^{\frac{N-\alpha}N}
\left(\|D^{\nu} \xi_k\|_{\frac{2N}\alpha}^{2} + \|D^{\nu} \eta_k\|_{\frac{2N}\alpha}^{2} \right)\\
&\le C\ve^{\frac{N-\alpha}N}.
\end{aligned}
$$

For $A_{2,3}$ we apply H\"older  inequality and the fact that $\|u_k\|_{2_{\alpha}^*}=1$ to get
$$
A_{2,3}\le C \left( \int_{B^c_{R_k}(z_k)}  |D^{\nu}\xi_k|^{\frac{2N}{\alpha}} +  |D^{\nu}\eta_k|^{\frac{2N}{\alpha}} \, dx   \right)^{\frac{\alpha}{N}}	.
$$
Now we estimate the last term, where the geometric result, Proposition \ref{FELICIDAD}, enters.
\begin{align*}
	 &\int_{B^c_{R_k}(z_k)} |D^{\nu} \xi_k|^{\frac{2N}{\alpha}}\, dx\\
	 &=\int_{R_k}^{\infty}s^{N-1}\int_{\mathbb{S}^{N-1}}\left[\int_{\mathbb{S}^{N-1}}\int_{\frac{s}{2}}^{2s}\left|\xi\left( \frac{s}{R_1}\phi -\frac{r}{R_1}\theta\right)\right|^2 \, d\mu(\theta)\frac{dr}{r^{\alpha+1}}\right]^{\frac{N}{\alpha}}d\phi ds\\
	&\le C\int_{R_k}^{\infty}s^{N-1}\int_{\mathbb{S}^{N-1}}\left[\int_{\mathbb{S}^{N-1}}\int_{\frac{s}{2}}^{2s}\chi_{\left\{ | \frac{s}{2R_1}\phi +\frac{r}{2R_1}\theta|<1\right\}}(\theta) \, d\mu(\theta)\frac{dr}{r^{\alpha+1}}\right]^{\frac{N}{\alpha}}d\phi ds\\
	&\le \frac{C}{R_1^{N}}\int_{R_k}^{\infty}s^{N-1}\int_{\mathbb{S}^{N-1}}
\left[\int_{\mathbb{S}^{N-1}}\int_{\frac{s}{4R_1}}^{\frac{s}{R_1}}\chi_{\left\{ | \frac{s}{2R_1}\phi +\rho\theta|<1\right\}}(\theta) \, d\mu(\theta)\frac{d\rho}{\rho^{\alpha+1}}\right]^{\frac{N}{\alpha}}d\phi ds\\
	&\le C\int_{R_k}^{\infty}t^{N-1}\int_{\mathbb{S}^{N-1}}\left[\int_{\mathbb{S}^{N-1}}
\int_{\frac{t}{2}}^{2t}\chi_{\left\{ | t\phi +\rho\theta|<1\right\}}(\theta) \, d\mu(\theta)\frac{d\rho}{\rho^{\alpha+1}}\right]^{\frac{N}{\alpha}}d\phi dt\\
	&\le C\int_{R_k}^{\infty}t^{N-1}\frac{1}{t^{2N-1}} dt=\frac{C}{R_k^{N-1}},
\end{align*}
for $k\ge k_0$. Analogously,
$$
\int_{B^c_{R_k}(z_k)} |D^{\nu} \eta_k|^{\frac{2N}{\alpha}}\, dx \le \frac{C}{R_k^{N-1}},
$$
so that  $A_{2,3}\le C\ve.$ In summary $|A_2|+|A_3|+|A_4| \le \delta(\ve)\to0$.
Therefore we conclude
\begin{align*}
S=\lim_{k\to \infty}\mathcal{E}(u_k, u_k) &\ge \liminf_{k\to \infty} \left[\mathcal{E}(u_k^1, u_k^1)+\mathcal{E}(u_k^2,u_k^2)\right]-\delta(\ve)\\
&\ge S\liminf_{k\to \infty}\left[(\|u_k^1\|_{2_{\alpha}^*}^{2_{\alpha}^*})^{\frac{N-\alpha
}{N}}+(\|u_k^2\|_{2_{\alpha}^*}^{2_{\alpha}^*})^{\frac{N-\alpha}{N}}\right]-\delta(\ve)\\
&\ge S \left((\varrho-\ve)^{\frac{N-\alpha}{N}}+(1-\varrho-\ve)^{\frac{N-\alpha}{N}}\right)-\delta(\ve)\\
&\underset{\ve \to 0}{\longrightarrow} S \left(\varrho^{\frac{N-\alpha}{N}}+(1-\varrho)^{\frac{N-\alpha}{N}}\right)>S\\
\end{align*}
for every $0<\rho<1$. We get a contradiction and dichotomy cannot occur.
	
Finally, compactness (C) occurs, that is, there exists a sequence $\{ y_k\}_{k\in \N}$ such that $\tilde{u}_k(x)=u_k(x-y_k)$
		is a minimizing sequence of \eqref{S-mu} verifying
	\begin{align*}
		\forall \ve>0\quad \exists \, R>0\quad \text{ s.t. }  \int_{B_R(y_k)} |\tilde{u}_k|^{2_{\alpha}^*} \, dx\ge 1-\ve.
	\end{align*}
	
	Let us rename again $\{ \tilde{u}_k\}_{k\in\N}$ as $\{u_k\}_{k\in \N}$, and let $u$ be the weak limit of $\{ u_k\}_{k\in \N}$ in $\dot{H}^{\frac\alpha2}(\R^N)$, $L^{2_{\alpha}^*}(\R^N)$. We prove by contradiction
that $u\not \equiv 0$. In fact assuming $u\equiv0$,  and applying Theorem \ref{concentracion-compacidad-RN}, we obtain that, up to taking a subsequence,  there exist measures $\eta,\varrho$ such that
	\begin{equation}\label{eta}
	|D^{\nu}u_k|^2dx \stackrel{*}{\cd} \eta \ge \sum_{i\in I} \eta_i \delta_{x_i} \text{ in } \mathcal{M}(\R^N),
	\end{equation}
	\begin{equation}\label{rho}
	|u_k|^{2_{\alpha}^*}dx \stackrel{*}{\cd} \varrho=\sum_{i\in I} \varrho_i \delta_{x_i} \text{ in } \mathcal{M}(\R^N)\color{blue},
	\end{equation}
as $k\to\infty$, and for every $i\in I$ it holds
	\begin{equation}\label{numeritos}
S^{\frac{1}{2}}\varrho_i^{\frac{1}{2_{\alpha}^*}} \le \eta_i^{\frac{1}{2}}.
	\end{equation}
	where $x_i\in \R^N$, $\varrho_i, \eta_i$ are positive constants, and the set $I$ is at most countable.
	
	Observe that from \eqref{rho}, \eqref{numeritos}, and \eqref{eta}, we get
	$$
	\varrho(\R^N)^{\frac{N-\alpha}{2N}}=\sum_{i\in I} {\varrho_i}^{\frac{N-\alpha}{2N}} \le S^{-\frac{1}{2}}  \sum_{i\in I}\eta_i^{\frac{1}{2}}  \le  S^{-\frac{1}{2}} \eta(\R^N)^{\frac{1}{2}}.
	$$
	Moreover, thanks to \eqref{rho}, \eqref{numeritos}, and \eqref{eta} a reverse H\"older inequality is satisfied, see \eqref{reverse-Holder}. Hence, by \cite[Lemma I.2]{Lions-limit-case-PartI}, the set $I$ reduces to a single point $x_0 \in \R^N$, and $\varrho=\zeta \delta_{x_0}= (\zeta^{\frac{N-\alpha}{N}} S)^{-1} \eta$, for some non-negative constant $\zeta>0$. Moreover, with the relations above we can prove $\zeta\ge 1$.
	
	Now, from \eqref{anti-vanishing}, we know that
	$$
	\frac{1}{2}= Q_{k}(1)\ge \int_{B_1(x_0)} |u_k|^{2_{\alpha}^*} \, dx.
	$$
	Taking the limit, we obtain
	$$
	\frac{1}{2} \ge \varrho(B_1(x_0)) = \zeta \delta_{x_0}(B_1(x_0)) = \zeta \ge 1,
	$$
	which is a contradiction. So that, $u\not \equiv 0$.

Let us show that $u_k \to u$ strongly in $L^{2_{\alpha}^*}(\R^N)$ as $k\to\infty$.  From the previous step we know that $\beta:=\int_{\R^N}|u|^{2_{\alpha}^*} \, dx\in (0,1]$; the goal is to prove $\beta=1$. Assuming $\beta<1$, and applying Theorem \ref{concentracion-compacidad-RN}, we to obtain
\begin{equation}\label{aux}
	\sum_{i\in I} \varrho_i=1-\beta>0, \qquad\int_{\R^N}|D^{\nu} u|^2\, dx= \mathcal{E}(u,u) \le S-\sum_{i\in I}\eta_i.
\end{equation}
Then by  \eqref{numeritos}, again the strict subadditivity of the power $t^{\frac{N-\alpha}N}$, and recalling Lemma~\ref{vale-siempre},  we get
	\begin{align*}
		S_{\beta} &\le \mathcal{E}(u,u)\le S -\sum_{i\in I} \eta_i \le S\left(1 -\sum_{i\in I}\varrho_i^{\frac{N-\alpha}{N}}\right)\\
		&<S\left( 1- \sum_{i\in I} \varrho_i\right)^{\frac{N-\alpha}{N}}=S(1-(1-\beta))^{\frac{N-\alpha}{N}}=S_{\beta},
	\end{align*}
	which is a contradiction. We conclude the convergence in norm, which together with the weak convergence gives  $u_k \to u$ strongly in $L^{2_{\alpha}^*}(\R^N)$ as $k\to\infty$, and $u$ is the desired minimizer.
\end{proof}

\section*{Acknowledgments}
This research has received funding from the European Union's Horizon 2020 research and innovation program under the Marie Skłodowska-Curie grant agreement No 777822 and from the MTM2017-87596-P Project of  Spanish Ministry of Economy. F.\,Quir\'os is also supported by the ICMAT–Severo Ochoa grant CEX2019-000904-S from the Spanish Ministry of Economy; A.\,Ritorto is also supported by the Dutch Research Council NWO through the project TOP2.17.012.

\bibliographystyle{amsplain}
\bibliography{biblio}

\end{document}